\documentclass[11pt]{amsart}
\usepackage{amssymb}
\usepackage{amsmath}
\usepackage{tikz}
\usetikzlibrary{graphs, decorations.pathmorphing, decorations.pathreplacing}
\usepackage{hyperref}
\usepackage{geometry}

\hypersetup{
	pdfstartview={XYZ null null 1.00}, 
	pdfpagemode=UseNone, 
	colorlinks,
	breaklinks, 
	linkcolor=blue,
	urlcolor=blue, 
	anchorcolor=blue,
	citecolor=blue
}

\newtheorem{theorem}{Theorem}[section]
\newtheorem{question}{Question}
\newtheorem{lemma}[theorem]{Lemma}
\newtheorem{prop}[theorem]{Proposition}
\newtheorem{cor}[theorem]{Corollary}

\newtheorem{claim}[theorem]{Claim}
\newtheorem{subclaim}[theorem]{Subclaim}

\theoremstyle{definition}
\newtheorem{definition}[theorem]{Definition}
\newtheorem{remark}[theorem]{Remark}

\newcommand{\ran}{\textnormal{ran}}
\newcommand{\inv}{^{-1}}
\newcommand{\Z}{\mathbb{Z}}
\newcommand{\F}{\mathbb{F}}
\newcommand{\N}{\mathbb{N}}

\newcommand{\B}{\mathcal{B}}

\def\acts{\curvearrowright}
\newcommand{\CONT}{\mathtt{CONTINUOUS}}
\newcommand{\BOREL}{\mathtt{BOREL}}
\newcommand{\MEAS}{\mathtt{MEASURE}}
\newcommand{\BAIRE}{\mathtt{BAIREMEAS}}

\newcommand{\EXISTS}{\mathtt{EXISTS}}
\newcommand{\BC}{\mathtt{BAIRE}}
\newcommand{\Sig}{\mathbf{\Sigma}}
\newcommand{\Del}{\mathbf{\Delta}}
\newcommand{\boldPi}{\mathbf{\Pi}}

\title{LCLs in the Borel Hierarchy}
\author{Felix Weilacher}
\thanks{The author is supported by the NSF under award number DMS-2402064.}

\begin{document}

\begin{abstract}
A locally checkable labeling problem (LCL) on a group $\Gamma$ asks one to find a labeling of the Cayley graph of $\Gamma$ satisfying a fixed, finite set of “local” constraints. Typical examples include proper coloring and perfect matching problems. 
In descriptive combinatorics, one often considers the existence of solutions to LCLs in the setting of descriptive set theory. For example, given a free action of $\Gamma$ on a Polish space $X$, we might be interested in solving a given LCL on each orbit in a continuous, Borel, measurable, etc. way. 


In an attempt to understand more finely the gap between Borel and continuous combinatorics, we consider the existence of Baire class $m$ solutions to LCLs. For all $n > 1$ and $m \in \omega$, we produce an LCL on $\mathbb{F}_n$ which always admits Baire class $m+1$ solutions, but not necessarily Baire class $m$ solutions. 
\end{abstract}

\maketitle

\section{Introduction}


This paper is about \textit{descriptive combinatorics}, in which certain classical problems from combinatorics are studied in the context of descriptive set theory. Some problems which have received much attention in the field are those in which a global labeling of some structure, typically a graph, must be found subject to ``local'' constraints. 

In this paper, our graphs will be those induced by actions of finitely generated groups, and in this setting the following formalizes the class of problems described above. These are sometimes called \textit{domino problems} on a group. 

\begin{definition}[\cite{NaorStock}]\label{def:LCL}
    Let $\Gamma$ be a countable group. A \textit{locally checkable labeling problem}, or \textit{LCL}, on $\Gamma$ is a triple of the form $\Pi = (\Lambda,S,\mathcal{A})$, where
    \begin{itemize}
        \item $\Lambda$ is a finite set (the \textit{labels}).
        \item $S \subseteq \Gamma$ is finite (the \textit{window}).
        \item $\mathcal{A} \subseteq \Lambda^S$. (the \textit{constraints}).
    \end{itemize}
    Let $\Gamma \acts X$ be a free action. A \textit{solution} to $\Pi$ on $X$, also called a \textit{$\Pi$-labeling} of $X$, is a function $f : X \to \Lambda$ such that for all $x \in X$,
    \[ (\gamma \mapsto f(\gamma \cdot x)) \in \mathcal{A}.\]
\end{definition}

Intuitively, $\Pi$ is an instance-solution problem, where solving $\Pi$ on an instance $\Gamma \acts X$ amounts to labeling the points in $X$ with $\Lambda$ subject to some constraints which are local in the sense that they can be checked by looking at the window $S \cdot x$ around each point $x$. 

The archetypal example of an LCL is \textit{proper coloring}. Recall that for $k$ a cardinal, a (proper) $k$-coloring of a graph $G$ with vertex set $X$ is a function $f : X \to k$ assigning different values (``\textit{colors}'') to adjacent vertices. If $\Gamma \acts X$ and $S \subseteq \Gamma$, we let $G(X,S)$ denote the associated \textit{Schreier graph}. Its vertex set is $X$ and two points $x,y \in X$ are adjacent if and only if $\exists \gamma \in S, \gamma \cdot x = y$ or vice versa.
For $k \in \omega$, the problem of proper $k$-coloring Schreier graphs of $\Gamma$ with respect to $S$ is an LCL of the form $\Pi = (k,\{1\} \cup S,\mathcal{A})$, with 
\[ \mathcal{A} = \{f : \{1\} \cup S \to \Lambda \mid \forall \gamma \in S \setminus \{1\}, f(1) \neq f(\gamma)\} \]
in the sense that for any free $\Gamma \acts X$, $k$-colorings of $G(X,S)$ are exactly solutions to $\Pi$ on $X$. 

LCLs on groups have long been studied in the context of logic and definability. The \textit{domino problem} for group $\Gamma$, for example, asks whether the set of LCLs $\Pi$ having a solution on the left translation action $\Gamma \acts \Gamma$ (equivalently, on any free action of $\Gamma)$ is decidable. This goes back to the influential work of Berger on Wang tilings, where $\Gamma = \Z^2$ \cite{berger1966undecidability}. We instead focus on the definability of solutions to LCLs themselves.

\begin{definition}\label{def:borel_coloring}
    Let $\Pi$ be an LCL on $\Gamma$.
    Let $\Gamma \acts X$ be a free Borel action on a standard Borel space $X$. By a \textit{Borel} $\Pi$-labeling of $X$ we simply mean a $\Pi$-labeling of $X$ which is Borel as a function $X \to \Lambda$. 

    We define measurable, Baire measurable, continuous, etc. $\Pi$-labelings similarly, given the appropriate structures on $X$ to make these notions meaningful. 
\end{definition}

The study of definable solutions to specific LCLs on specific actions is often important to questions in the dynamics of countable group actions which are not a priori about LCLs. For example, results about (definable) equidecompositons, such as versions of the Banach-Tarski paradox or Tarski's circle-squaring problem, often depend on the analysis of (definable) matchings in graphs \cite{GMP, marks-unger,marks-unger-BT}. As another example, in the most recent progress on Weiss' question -- whether Borel actions of amenable groups generate hyperfintie equivalence relations -- an important role is played by the Borel colorability of a certain auxiliary graph, and the key new tool, asymptotic dimension, can also be presented as an LCL \cite{dimension}.

It is therefore important to understand when an LCL on a group $\Gamma$ admits definable solutions for \textit{any} appropriate action of $\Gamma$. Versions of this question give rise to what might be called ``complexity classes''.

\begin{definition}
    Fix a countable group $\Gamma$ and an LCL $\Pi$ on $\Gamma$. 
    \begin{itemize}
        \item We say $\Pi \in \EXISTS(\Gamma)$ if for any free action $\Gamma \acts X$, $X$ admits a $\Pi$-labeling. Equivalently, $\Gamma$ admits a $\Pi$-labeling with respect to the left multiplication action.
        
        \item We say $\Pi \in \BOREL(\Gamma)$ if for any free Borel action on a standard Borel space $\Gamma \acts X$, $X$ admits a Borel $\Pi$-labeling. 
        
        \item We say $\Pi \in \CONT(\Gamma)$ if for any free continuous action on a 0-dimensional Polish space $\Gamma \acts X$, $X$ admits a continuous $\Pi$-labeling. 
        
        \item We say $\Pi \in \MEAS(\Gamma)$ if for any free Borel action on a standard probability space $\Gamma \acts (X,\mu)$, $X$ admits a $\mu$-measurable $\Pi$-labeling.

        \item We say $\Pi \in \BAIRE(\Gamma)$ if for any free Borel action on a Polish space $\Gamma \acts X$, $X$ admits a Baire measurable $\Pi$-labeling. 
    \end{itemize}
\end{definition}

The resulting ``complexity theory'' gives a natural way of classifying problems by their difficulty from the point of view of descriptive combinatorics. 

We now consider the case where $\Gamma$ is a free group, $\Gamma = \F_n$ for $n \leq \omega$. For $n = 1$, that is, $\Gamma = \Z$, it turns out that essentially all of the nontrivial descriptive complexity classes coincide, consisting of the problems whose solutions do not enforce any sort of periodicity \cite{GR_paths}. When $n > 1$ the situation is more interesting. Since all the non abelian free groups are subgroups of eachother, we can fix $n = 2$ here without losing generality (see Section \ref{sec:shift}).
This case was systematically considered in \cite{BCGGRV}, where the strict hierarchy pictured in Figure \ref{fig:linear_hierarchy} is obtained. (This combines the results from many other papers, \cite{kst, conley2016brooks, conley.miller.toast, conley2013measurable, Marks}.)

\begin{figure}\label{fig:linear_hierarchy}
\begin{tikzpicture}
                \node (1) at (0, 0) {$\CONT(\F_n)$};
                \node (2) at (0, 1.5) {$\BOREL(\F_n)$};
                \node (3) at (0, 3) {$\MEAS(\F_n)$};
                \node (4) at (0, 4.5) {$\BAIRE(\F_n)$};
                \node (5) at (0,6) {$\EXISTS(\F_n)$};
                \node (6) at (0,7.5) {All LCLs};

                \node[align = right]  at (-3, 0) {$2n+1$-coloring};
                \node[align = right]  at (-3, 3) {$2n$-coloring};
                \node[align = right]  at (-4.2, 4.5) {3-coloring (if $n > 2$) \\ paradoxical decomposition};
                \node[align = right]  at (-3,6) {2-coloring};
                \node[align = right] at (-3,7.5) {1-coloring};
                
                \draw[->] (1) -- (2) node[midway, right] {$\subseteq$};
                \draw[blue,->] (2) -- (3) node[midway, right] {$\subsetneq$};
                \draw[blue,->] (3) -- (4) node[midway, right] {$\subsetneq$};
                \draw[blue,->] (4) -- (5) node[midway, right] {$\subsetneq$};
                \draw[blue,->] (5) -- (6) node[midway, right] {$\subsetneq$};
\end{tikzpicture}
\caption{The complexity classes introduced so far for $\F_n$, $n > 1$. Arrows between classes denote inclusion, and blue arrows indicate that the inclusion is strict. To the left of each class is an LCL belonging to that class but no lower one. Each $k$-coloring problem refers to the Schreier graphs generated by the usual generating set for $\F_n$. 
}

\end{figure}
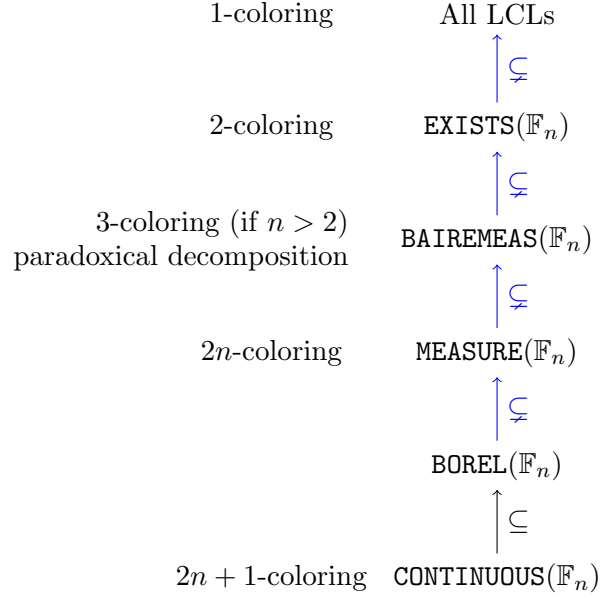

Let us point out the surprising inclusion $\MEAS(\F_n) \subsetneq \BAIRE(\F_n)$, proved originally in \cite{BCGGRV}. The rest of the inclusions in the diagram of course hold for any group, but it is interesting that here they are all strict. 

This paper is primarily concerned with the gap between $\CONT(\F_n)$ and $\BOREL(\F_n)$. This interval is naturally stratified by the \textit{Baire hierarchy}. Note that some authors use a different definition which is off by one for infinite $\alpha$. For this definition see \cite{louveau_notes}.

\begin{definition}
    Let $X$ and $Y$ be Polish spaces, $f : X \to Y$, and $\alpha < \omega_1$. We say that $f$ is \textit{Baire class $\alpha$}, written $f \in \B_\alpha$, if for any open $U \subseteq Y$, $f\inv(U) \in \Sig_{1+\alpha}(X)$. That is, $f$ is \textit{$\Sig_{1+\alpha}$-measurable}
\end{definition}

Thus, $\B_0$ is the set of continuous functions and 
$\bigcup_{\alpha < \omega_1} \B_\alpha$ is the set of Borel functions. For $\alpha < \omega_1$, Let us say that an LCL $\Pi$ on $\Gamma$ is in the class $\BC_\alpha$ if for any free continuous action on a 0-dimensional Polish space $\Gamma \acts X$, $X$ admits a Baire class $\alpha$ $\Pi$-labeling. Thus, $\BC_0(\Gamma) = \CONT(\Gamma)$. Of course, $\BC_\alpha(\Gamma) \subseteq \BC_{\beta}(\Gamma)$ if $\alpha \leq \beta$. 
Note that since the set of labels $\Lambda$ of an LCL $\Pi = (\Lambda,S,\mathcal{A})$ is finite, a $\Pi$-labeling $f : X \to \Lambda$ is Baire class $\alpha$ if and only if for each $\lambda \in \Lambda$, $f\inv(\lambda)$ is $\Del^0_{1+\alpha}$. 
This also explains the 0-dimensionality restriction. Without it, a space $X$ might have no nontrivial $\Del^0_1$, that is, clopen, sets, making the class $\CONT(\Gamma) = \BC_0(\Gamma)$ uninteresting. 

Our main result is that the first $\omega$ levels of this complexity hierarchy are distinct when $\Gamma = \F_2$.

\begin{theorem}\label{thm:main_hierarchy}
    For each $n \in \omega$, $\BC_n(\F_2) \subsetneq \BC_{n+1}(\F_2)$.
\end{theorem}

By Corollary \ref{cor:subgroup_LCL} the result also applies to all of the other non abelian free groups.
It may initially seem unsatisfying that this result does not extend to all $\alpha < \omega_1$. However, note that there are only countably many LCLs! We will also see in the next section that there is a \textit{universal} action $\Gamma \acts X$ of any group $\Gamma$ in this setting -- that is, one such that $X$ admits a continuous, Borel, or Baire class $\alpha$ solution to some LCL $\Pi$ on $\Gamma$ if and only if $\Pi$ is in the corresponding complexity class. This implies:

\begin{prop}
    Let $\Gamma$ be a countable group.
    \[ \BOREL(\Gamma) = \bigcup_{\alpha < \omega_1} \BC_\alpha(\Gamma).\]
\end{prop}

Thus the hierarchy $\BC_\alpha(\F_2)$ cannot increase strictly forever, and the following question is well defined. 

\begin{question}\label{q:main}
    What is the least $\alpha < \omega_1$ such that $\BC_\alpha(\F_2) = \BOREL(\F_2)$? 
\end{question}

We conjecture that the answer is $\omega$. In fact we conjecture that $\BOREL(\F_2) = \bigcup_{n \in \omega} \BC_n(\F_2)$. 

This paper is not the first to consider the existence of solutions to LCLs at given levels of the Baire hierarchy. First, the continuous setting has naturally received much attention, see e.g. \cite{gjks}. At higher levels there are a series of interesting results of Lecomte and Zeleny concerning Borel graphs which are simple as subsets of the plane and Borel 2-colorable, but which don't have finite or even countable Baire-class $\alpha$ colorings for a given $\alpha$ \cite{LZ_closed,LZ_rectangles}.

However, this paper does seem to be the first to consider levels above continuous in the framework of LCLs on groups. The challenges here are somewhat different in flavor from those in Lecomte and Zeleny's work. While they need to find Borel graphs which are complicated in the appropriate ways, we are given some straightforward graphs, and are interested in whether combinatorial problems on those graphs can require complicated solutions. To highlight this difference, we point to the fact that for any infinite group $\Gamma$ with generating set $S$, proper 2-coloring with respect to $S$ is not in $\BOREL(\Gamma)$ or even $\BAIRE(\Gamma)$ \cite{weilacher2020marked}, while Lecomte and Zeleny's graphs are Borel 2-colorable.

\section{The Bernoulli shift}\label{sec:shift}

In this preliminary section we introduce the aforementioned universal action which explains some of the statements in the introduction

\begin{definition}
    Let $\Gamma$ be a countable group and $X$ a set. The \textit{Bernoulli shift} on $X$ is the action of $\Gamma$ on the set $X^\Gamma$ given by $\gamma(x)(\delta) = x(\gamma\inv\delta)$ for $x \in X^\Gamma$ and $\gamma,\delta \in \Gamma$. When $X$ is a topological space, we endow $X^\Gamma$ with the product topology, which makes this action continuous. 

    The \textit{free part} of the Bernoulli shift on $X$, denoted $F(\Gamma,X)$, is the set of points of the Bernoulli shift with trivial stabilizer. This is an invariant set and so inherits the action.
\end{definition}

Note that when $X$ is a Hausdorff space $F(\Gamma,X)$ is $G_\delta$. In particular it is Polish if $X$ is. It is also dense if $|X| > 1$. 

\begin{prop}
    Let $\Gamma$ be a countable group acting continuously on a 0-dimensional second countable Hausdorff space $X$. Then there is a continuous $\Gamma$-equivariant injective map $f : X \to (2^\omega)^\Gamma$. If $\Gamma \acts X$ is free there is such a map $f : X \to F(\Gamma,2^\omega)$. 
\end{prop}

\begin{proof}
    Let $g : X \to 2^\omega$ be a continuous injection. (For example, let $g(x)(n)$ record whether $x$ is in the $n$-th member of some fixed countable clopen basis for $X$.) Then let $f(x)(\gamma) = g(\gamma\inv x)$. The second part is automatic by equivariance. 
\end{proof}

\begin{prop}
    Let $\Gamma$ be a countable group with a Borel action on a standard Borel space $X$. Then there is a Borel $\Gamma$-equivariant injective map $f : X \to (2^\omega)^\Gamma$. If $\Gamma \acts X$ is free there is such a map $f : X \to F(\Gamma,2^\omega)$. 
\end{prop}

\begin{proof}
    Let $g : X \to 2^\omega$ be a Borel injection. Then repeat the previous proof. 
\end{proof}

\begin{cor}\label{cor:shift_LCL_universal}
    Let $\Gamma$ be a countable group and $\Pi$ an LCL on $\Gamma$. $\Pi \in \BOREL(\Gamma)$ if and only if $F(\Gamma,2^\omega)$ admits a Borel $\Pi$-labeling. Likewise for the classes $\BC_\alpha(\Gamma)$.
\end{cor}

Recall that this made Question \ref{q:main} well-defined. 
It also allows us to treat all the non-abelian free groups as equivalent from this perspective via the following observations.

\begin{prop}
    Suppose $\Delta \leq \Gamma$ are countable groups. There is a continuous $\Delta$-equivariant map $f : (2^\omega)^\Delta \to (2^\omega)^\Gamma$ with $f[F(\Delta,2^\omega)] \subseteq F(\Gamma,2^\omega)$. 
\end{prop}

\begin{proof}
    Let 
    \[ f(x)(\gamma) = \begin{cases}
        1^\frown x(\gamma) & \textnormal{if } \gamma \in \Delta \\ 0^\omega & \textnormal{else}
    \end{cases}, \]
    where $0^\omega \in 2^\omega$ denotes the constant 0 sequence, and $1^\frown x \in 2^\omega$ denotes the sequence $x$ with a 1 added to the front. This is clearly $\Delta$-equivariant, injective, and continuous. 

    Suppose $x \in F(\Delta,2^\omega)$ and $\gamma \in \Gamma$ with $\gamma f(x) = f(x)$. If $\gamma \notin \Delta$ then $(\gamma f(x))(\gamma) = f(x)(1) = 1^\frown x(1)$ while $f(x)(\gamma) = 0^\omega$, so these cannot be equal as their first bits differ. Thus $\gamma \in \Delta$, in which case $\gamma = 1$ by equivariance and injectivity. 
\end{proof}

Observe that if $\Delta \leq \Gamma$ and $\Pi = (\Lambda,S,\mathcal{A})$ is an LCL on $\Delta$, then $\Pi$ is also an LCL on $\Gamma$. Furthermore if $\Gamma \acts X$ is a free action and $f : X \to \Lambda$ is a function, then $f$ is a $\Pi$-labeling with respect to $\Gamma$ if and only if it is a $\Pi$-labeling with respect to the induced $\Delta$-action. 

\begin{cor}\label{cor:subgroup_LCL}
    Let $\Delta \leq \Gamma$ be countable groups and $\Pi$ an LCL on $\Delta$. Then $\Pi \in \BOREL(\Delta) \Leftrightarrow \Pi \in \BOREL(\Gamma)$, and likewise for the classes $\BC_\alpha$. 
\end{cor}

\begin{proof}
    The forward direction is immediate from the previous paragraph. On the other hand, if $F(\Gamma,2^\omega)$ has a Baire-class-$\alpha$ $\Pi$-labeling, then by the proposition so does $F(\Delta,2^\omega)$, so we are done by Corollary \ref{cor:shift_LCL_universal}. 
\end{proof}

\section{Complexity and 2-colorings}

Before proving the main theorem, we will need some understanding of how one of the most basic ``difficult'' LCLs, proper 2-coloring of $\Z$-actions, interacts with complexity. In this section we fix the generating set $S = \{1\}$ for $\Z$, and proper 2-coloring refers to the proper 2-coloring problem for $\Z$-actions with respect to this generating set, as defined in the example following Definition \ref{def:LCL}. It is well known that this LCL is not in $\BOREL(\Z)$ or even $\BAIRE(\Z)$ or $\MEAS(\Z)$ \cite{kst}. In fact something a bit stronger is true. 

\begin{definition}
    Let $\mathcal{I}$ be the $\sigma$-ideal consisting of $\Z$-invariant Borel sets $B \subseteq F(\Z,2^\omega)$ such that $B$ admits a Borel 2-coloring. 
\end{definition}

The hardness results mentioned above proceed by showing that members of $\mathcal{I}$ must be very small in the appropriate way. For instance, they must be meager with respect to any compatible Polish topology on $2^\omega$ and null with respect to any non-Dirac Borel probability measure on $2^\omega$. We include a proof of the former since we will need it later.

\begin{prop}\label{prop:2_color_meager}
    Let $A \in \mathcal{I}$ and $\tau$ a Polish topology on $2^\omega$ generating the same Borel sets as the usual topology. Then $A$ is meager in $(2^\omega)^\Z$ with respect to the product topology of $\tau$. 
\end{prop}

Note that $F(\Z,2^\omega)$ is dense $G_\delta$ in $(2^\omega)^\Z$, and so it does not matter which of these spaces we work in. 

\begin{proof}
    Let $c : A \to 2$ be a Borel 2-coloring, and suppose to the contrary $A$ is nonmeager. Without loss of generality $c\inv(0)$ is nonmeager. Then, since it is Borel, it is comeager on some basic open set of the form $U = \prod _{n \in \Z} U_n$, where $U_n \in \tau$ is nonempty open and $U_n = 2^\omega$ for $|n| > N$ for some $N \in \N$. By choice of $N$, the translate $(2N+1)U$ meets $U$, and then $c\inv(0) \cap (2N+1)c\inv(0)$ is comeager in the nonempty open set $(2N+1)U \cap U$. However, by the way 2-colorings work we have $kc\inv(0) = c\inv(1)$ for any odd integer $k$, so the former intersection is empty.
\end{proof}

The main result of this section, which may be of independent interest, says that members of $\mathcal{I}$ must also be negligible in terms of their ability to affect complexity. 

\begin{theorem}\label{thm:2_color_complexity}
    Let $0 <\alpha < \omega_1$. There is a $\Sig^0_\alpha$-complete set $U_\alpha \subseteq 2^\omega$ with the following property: For all $A \in \mathcal{I}$, $F(\Z,U_\alpha) \setminus A$ is $\boldPi^0_{\alpha+1}$-hard as a subset of $F(\Z,2^\omega)$. 
\end{theorem}

\begin{remark}
\begin{enumerate}
    \item In the proof of Theorem \ref{thm:main_hierarchy} we will only need this for finite $\alpha$.
    \item Since $F(\Z,U_\alpha) = F(\Z,2^\omega) \cap (U_\alpha)^\Z$ and $F(\Z,2^\omega)$ is $G_\delta$ in $(2^\omega)^\Z$, it is not important in the statement above whether we use $F(\Z,U_\alpha)$ and $F(\Z,2^\omega)$ or $(U_\alpha)^\Z$ and $(2^\omega)^\Z$ in their respective places. 
    \item It is also not important that we use a particular $\Sig^0_\alpha$-complete set; the theorem remains true if $U_\alpha$ is replaced by any $\Sig^0_\alpha$-hard set. This introduces some annoyances in maintaining freeness for small values of $\alpha$, so we content ourselves with this superficially weaker statement. 
\end{enumerate}
\end{remark}

Our proof will make use of machinery developed by Matrai \cite{matrai2007covering} for proving lower bounds in the Borel hierarchy using Baire category criteria. As part of this machinery, Matrai constructs pairs of subsets of $2^\omega$ and topologies on $2^\omega$ refining the usual one which he calls \textit{Topological Hurewicz test paris}. We will not give the general definition of these test pairs, but just collect in the following Lemma the aspects of the theory in the succesor case which we will use. We start with the sets, which are canonical. 

\begin{definition}\label{def:canonical_sets}
    We define pairs of sets $P_\alpha \subseteq X_\alpha$ by induction on $\alpha < \omega_1$. Let $X_0 = 2$ and $P_0 = \{0\} \subseteq X_0$. 
    
    For $0<\alpha < \omega_1$, fix a cofinal sequence $(\beta_i)_{i \in \omega}$ in $\alpha$ in the case where $\alpha$ is a limit, and let each $\beta_i$ be the predecessor of $\alpha$ in the case where $\alpha$ is a sucessor. Then define
    \begin{itemize}
        \item $X_{\alpha} = \prod_{i < \omega} X_{\beta_i}$.
        \item $P_\alpha = \{x \in X_\alpha \mid \forall i < \omega,\ x(i) \not\in P_{\beta_i}$\}. 
    \end{itemize}
\end{definition}

Note that if we give $X_0$ the discrete topology and each $X_\alpha$ the product topology, then $X_\alpha$ is homeomorphic to $2^\omega$ for each $\alpha > 0$, and each $P_\alpha$ is $\boldPi^0_\alpha$-complete in $X_\alpha$. Call this topology on $X_\alpha$ $\sigma_\alpha$.

\begin{lemma}[Matrai \cite{matrai2007covering}]\label{lem:matrai_properties}
    For each $0 < \alpha < \omega_1$ there are Polish topologies $\sigma_\alpha \subseteq \tau_{\alpha}^< \subseteq \tau_\alpha$ with the following properties.
    \begin{enumerate}
        \item $P_\alpha$ is closed nowhere dense in $\tau_\alpha$. 
        \item $\tau_{\alpha}^<$ is the product topology with respect to the $\tau_{\beta_i}$'s.  
        \item If $B \subseteq X_\alpha$ is $\Sig^0_\alpha(\sigma_\alpha)$ and $\tau_\alpha^<$-nonmeager then it is $\tau_\alpha$-nonmeager.
    \end{enumerate}
\end{lemma}

Definition \ref{def:canonical_sets} is Definition 35 from \cite{matrai2007covering}. Then items (1)-(3) in the Lemma are from Definition 15(3), Proposition 36, and Corollary 23(2) respectively in \cite{matrai2007covering}. (3) is key and delivers on the promise of being able to obtain lower bounds using Baire category.

\begin{proof}[Proof of Theorem \ref{thm:2_color_complexity}]
    We identify $(X_\alpha,\sigma_\alpha)$ with $2^\omega$ with its usual topology and take $U_\alpha = X_\alpha \setminus P_\alpha$. We also identify $X_{\alpha+1} := X_\alpha^\omega$ with $X_\alpha^\Z$ by way of a bijection between $\omega$ and $\Z$. Let $A \in \mathcal{I}$. 

    By Lemma \ref{lem:matrai_properties}(3) together with Wadge determinacy, it suffices to show $F(\Z, U_\alpha) \setminus A$ is $\tau_{\alpha+1}^<$-comeager and $\tau_{\alpha+1}$-meager. For the former, recall that $\tau_{\alpha+1}^<$ is the product topology with respect to $\tau_\alpha$, so $A$ is meager in it by Proposition \ref{prop:2_color_meager}, while $U_\alpha^\Z$ is dense $G_\delta$ in it by Lemma \ref{lem:matrai_properties}(1) and then $F(\Z,U_\alpha)$ is dense $G_\delta$ in $U_\alpha^\Z$ as was observed in Section \ref{sec:shift}. For the latter, observe that $F(\Z,U_\alpha) \subseteq U_\alpha^\Z = P_{\alpha+1}$ and apply Lemma \ref{lem:matrai_properties}(1). 
\end{proof}


\section{Proof of the main theorem}

In this section we will prove the following by induction on $n$. 

\begin{theorem}\label{thm:main_strong}
    For each $0<n \in \omega$ there is an LCL $\Pi \in \BC_n(\F_{n+1})$ with label set $\Lambda$ and a subset $\Lambda_* \subseteq \Lambda$ such that for any Borel $\Pi$-labeling $c : F(\F_{n+1},2^\omega) \to \Lambda$, $c\inv(\Lambda_*)$ is $\boldPi^0_{n}$-hard. 
\end{theorem}

The latter property implies $\Pi \not\in \BC_{n-1}(\F_{n+1})$. This proves Theorem \ref{thm:main_hierarchy} since each $\F_{n+1}$ is a subgroup of $\F_2$, so we can apply Corollary \ref{cor:subgroup_LCL}.

\subsection{The base case}

In this section we prove the $n=1$ case of Theorem \ref{thm:main_strong}. This will serve as a warmup for the inductive step, containing many of the key ideas. 

Write $\F_2 = \langle a,b\rangle$. $\Pi$ will be the following LCL on $\F_2$: The label set will be $\Lambda = \{0,1,N,*\} \times \{R,B,G\}$. The window will be $\{a,b\}$. The constraints are the following:
\begin{itemize}
    \item The second coordinates give a 3-coloring of the $b$-orbits. (We think of the colors as red, blue, and green, hence the labels.) That is, $c(x)_2 \neq c(bx)_2$ for all $x$, the subscript 2 denoting projection onto the second coordinate. 
    \item The first coordinate is $*$ if and only if the second coordinate is $R$ or $B$. 
    \item The vertices with first coordinate $N$ are $a$-invariant. That is, for every $x$, $c(x)_1 = N \Leftrightarrow c(ax)_1 = N$. 
    \item The 0 and 1 labels give a partial 2-coloring of the $a$-orbits. That is, for every $x$, we cannot have $c(x)_1 = c(ax)_1 \in \{0,1\}$. 
\end{itemize}
In other words, given a 3-coloring of the $b$-orbits, there are 3 options for how the first coordinates on each $a$-orbit can look. If the orbit contains any $R$ or $B$ points, those must be labeled $*$ and the remaining $G$ vertices must be properly 2-colored by $\{0,1\}$. If the orbit consists of only $G$ points, we can either $2$-color the entire orbit with $\{0,1\}$ or label it all with $N$. These three cases are illustrated in the second, third, and first rows respectively of Figure \ref{fig:base_case}.

\begin{figure}
    \centering

    \begin{tikzpicture}[nodes={circle, draw, minimum size=5mm, inner sep=0pt}]
    \graph [nodes={circle, draw, minimum size=5mm, inner sep=2pt, empty nodes},
            edges={-}]
    {11[label=center:\small$N$, fill = green!50] --
     12[label=center:\small$N$, fill = green!50] --
     13[label=center:\small$N$, fill = green!50] -- 
     14[label=center:\small$N$, fill = green!50] -- 
     15[label=center:\small$N$, fill = green!50] -- 
     16[label=center:\small$N$, fill = green!50] -- 
     17[label=center:\small$N$, fill = green!50];
    
    1[label=center:\small$*$, fill = red!50] --
     2[label=center:\small$0$, fill = green!50] --
     3[label=center:\small$1$, fill = green!50] -- 
     4[label=center:\small$0$, fill = green!50] -- 
     5[label=center:\small$*$, fill = blue!50] -- 
     6[label=center:\small$1$, fill = green!50] -- 
     7[label=center:\small$*$, fill = red!50];

     21[label=center:\small$0$, fill = green!50] --
     22[label=center:\small$1$, fill = green!50] --
     23[label=center:\small$0$, fill = green!50] -- 
     24[label=center:\small$1$, fill = green!50] -- 
     25[label=center:\small$0$, fill = green!50] -- 
     26[label=center:\small$1$, fill = green!50] -- 
     27[label=center:\small$0$, fill = green!50];

     17 -- 7 --27;
     };
     \node (A)[above of = 17, fill = blue!50] {\small$*$};
     \node (B)[below of = 27, fill = red!50] {\small$*$};

     \draw (A) -- (17); 
     \draw (B) -- (27);

     \draw[->] (0, -3) -- (2, -3) node[draw = none, midway, above] {\small$a$}; 
    
    \draw[->] (7, -1) -- (7, 1) node[draw = none, midway, right] {\small$b$}; 
     \end{tikzpicture}
    \caption{A part of a Schreier graph of $\F_2$ with a $\Pi$-labeling in the base case $n = 1$. The $a$ and $b$ orbits are drawn horizontaly and vertically respectively. The second coordinates of each label are shown as the color of a vertex while the first coordinates are simply written on the vertex.}\label{fig:base_case}
\end{figure}
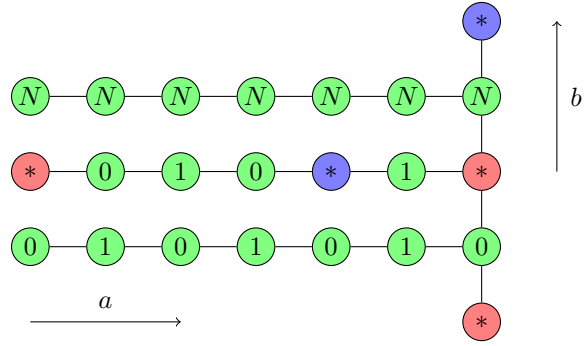

\begin{claim}\label{claim:base_case_bc1}
    $\Pi \in \BC_1(\F_2)$. 
\end{claim}

\begin{proof}
    By \cite{kst, bernshteyn2023distributed}, 3-coloring with respect to the usual generator is in $\CONT(\Z)$. Thus there is a continuous $c:F(\F_2,2^\omega) \to \{R,B,G\}$ which is a 3-coloring of the $b$-orbits. We will use $c$ for our second coordinates and construct a Baire-class 1 function $d:F(\F_2,2^\omega) \to \{0,1,N,*\}$ so that $(d,c)$ is a $\Pi$-labeling. 

    We are forced to put $d(x) = *$ if and only if $c(x) \in \{R,B\}$, so $d\inv(*)$ is clopen. Put $d(x) = N$ if and only if $\langle a \rangle x \subseteq c\inv(G)$. That is, any time an $a$-orbit is entirely green, we label it all $N$ as in the first row in Figure \ref{fig:base_case}. Then $d\inv(N)$ is closed. 

    Let $L$ be the set of remaining points. That is, $G$ points with some $*$ in their $a$-orbit. 
    $L$ is covered by the open sets 
    $L_- := \{x \in L \mid \exists n > 0, d(a^{-n} x) = *\} $ and 
    $L_+ := \{x \in L \mid \exists n > 0, d(a^n x) = *\} $. Now define 
    \begin{itemize}
        \item $L^0_- = \{x \in L_- \mid \textnormal{the least } n \textnormal{ witnessing } x \in L_- \textnormal{ is even}\}$.
        \item $L^1_- = \{x \in L_- \mid \textnormal{the least } n \textnormal{ witnessing } x \in L_- \textnormal{ is odd}\}$.
        \item $L^0_+ = \{x \in L_+ \mid \textnormal{the least } n \textnormal{ witnessing } x \in L_+ \textnormal{ is even}\}$.
        \item $L^1_+ = \{x \in L_+ \mid \textnormal{the least } n \textnormal{ witnessing } x \in L_+ \textnormal{ is odd}\}$.
    \end{itemize}
    For a given $n > 0$, the set of $x$ such that $n$ is the least witness to $x \in L_-$ is clopen since $d\inv(*)$ is clopen, and likewise for $L_+$. Thus each of these sets is open. Each is also clearly $a$-independent. It is then easy to check that we can set $d\inv(i) = L_-^i \cup (L_+^i \setminus L_-)$ for each $i$. 
\end{proof}

The proof of this Claim explains the choice of this LCL. We want to code a $\boldPi^0_1$-predicate like ``every point in my $a$-orbit is green'', and we add a label $N$ which is supposed to do this. While a local rule can make sure orbits labeled with $N$ do indeed contain only green points, there is no such way to force an entirely green orbit to use the $N$ label. Instead of forcing it, we strongly encourage the use of the label $N$ in this case by making the alternative a task (proper 2-coloring) which is very difficult on $a$-invariant sets (as reflected by Proposition \ref{prop:2_color_meager} and Theorem \ref{thm:2_color_complexity}), but which can still be performed easily on orbits that are not all green.

As suggested by this discussion, for the second half of the base case of Theorem \ref{thm:main_strong} we'll take $\Lambda_* = \{N\} \times \{R,B,G\}$, the set of points labeled $N$ in the first coordinate. (It would be equivalent to take $\Lambda_* = \{(N,G)\}$ since the other pairings are not allowed.)

\begin{claim}\label{claim:base_case_hardness}
    Let $(d,c) : F(\F_2,2^\omega) \to \{0,1,N,*\} \times \{R,B,G\}$ be a Borel $\Pi$-labeling. $d\inv(N)$ is $\boldPi^0_1$-hard. 
\end{claim}

An important remark regarding this Claim is that it is not even a priori clear that $d\inv(N)$ is nonempty! For instance, suppose we defined the same LCL on $\Z^2 = \langle a,b \mid ab = ba\rangle$. By a result of Gao, Jackson, Krohne, and Seward \cite{GJKS_borel} there is a Borel 3-coloring of $F(\Z^2,2^\omega)$ with respect to the generating set $\{a,b\}$. If we used this as our 3-coloring of the $b$-orbits, then there would be no $a$-orbit with only green points. 
Fortunately, a result of Marks says that $F(\F_2,2^\omega)$ has no Borel 3-coloring with respect to $\{a,b\}$ \cite{Marks}. 
The techniques used to show this, namely the following lemma, will also be key in our situation. 

\begin{lemma}[{\cite[Lemma 2.12, Remark 2.9]{marks2017uniformity}}]
    Let $\Delta,\Gamma$ be countable groups and $A \subseteq F(\Delta * \Gamma, 2^\omega \times 2^\omega)$ Borel. At least one of the following holds:
    \begin{itemize}
        \item There is a continuous $\Delta$-equivariant function $f:F(\Delta,2^\omega) \to F(\Delta * \Gamma,2^\omega \times 2^\omega)$ with $\ran(f) \subseteq A$ and moreover $(f(x)(1))_1 = x(1)$ for all $x$.
        \item There is a continuous $\Gamma$-equivariant function $f:F(\Gamma,2^\omega) \to F(\Delta * \Gamma,2^\omega \times 2^\omega)$ with $\ran(f) \cap A = \emptyset$ and moreover $(f(x)(1))_2 = x(1)$ for all $x$.
    \end{itemize}
\end{lemma}

We will not need the ``moreover'' parts until the next section, so for now we can ignore them and identify $2^\omega \times 2^\omega$ with $2^\omega$. 

\begin{subclaim}
    $d\inv(N) \neq \emptyset$.
\end{subclaim}

\begin{proof}
    Apply Marks' lemma with $\Delta = \langle a \rangle$, $\Gamma = \langle b \rangle$, and $A = c\inv(G)$. In the second case $c \circ f$ is a proper coloring of $F(\Z,2^\omega)$ using only $R$ and $B$, contradicting the fact that 2-coloring is not in $\BOREL(\Z)$. If $d\inv(N) = \emptyset$ then $A = d\inv(\{0,1\})$, so in the first case $d \circ f$ is a proper coloring of $F(\Z,2^\omega)$ using only $0$ and $1$, a contradiction again. 
\end{proof}

It is easy to see that a nonempty set with empty interior is $\boldPi^0_1$-hard, so the following completes the proof of Claim \ref{claim:base_case_hardness} and therefore our base case. 

\begin{subclaim}
    $d\inv(N)$ has empty interior. 
\end{subclaim}

\begin{proof}
    The action of $a$ on $F(\F_2,2^\omega)$ is generically ergodic; it has a dense orbit. This is easy to see since basic open sets in $(2^\omega)^{\F_2}$ constrain only finitely many coordinates. Therefore, since $d\inv(N)$ is $a$-invariant, if it had nonempty interior it would be comeager. Then $c\inv(G) \supseteq d\inv(N)$ would be comeager, but this cannot be the case since it is $b$-independent.
\end{proof}

\subsection{The inductive step}

In this section we prove the following, which clearly provides the inductive step for Theorem \ref{thm:main_strong}. (Take $\alpha = n \in \omega$ and $\beta = n+1$.)

\begin{theorem}
    Let $\Gamma$ be a countable group, $\alpha < \beta < \omega_1$, and $\Pi \in \BC_{\beta}(\Gamma)$ with label set $\Lambda$. Suppose there is a subset $\Lambda_* \subseteq \Lambda$ such that for any Borel $\Pi$-labeling $c : F(\Gamma,2^\omega) \to \Lambda$, $c\inv(\Lambda_*)$ is $\boldPi^0_{1+\alpha}$-hard. 

    Then there is an LCL $\Pi' \in \BC_{\beta + 1}(\Z * \Gamma)$ with label set say $\Lambda'$ and a subset $\Lambda_*' \subseteq \Lambda'$ such that for any Borel $\Pi$-labeling $c : F(\Z * \Gamma,2^\omega) \to \Lambda'$, $c\inv(\Lambda_*')$ is $\boldPi^0_{1+\alpha+1}$-hard. 
\end{theorem}

Note that we need $\alpha < \beta$ for the assumption to be non vacuous. 

The construction and the proof that it works will closely parallel those from the base case. Let $\Z = \langle a\rangle$. We will have $\Lambda' = \{0,1,N,*\} \times \Lambda$. The window will be $\{a\} \cup S$ where $S$ is the window of $\Pi$. The constraints are the following:
\begin{itemize}
    \item The second coordinate gives a $\Pi$-labeling of each $\Gamma$-orbit. 
    \item The first coordinate is $*$ if and only if the second coordinate is in $\Lambda_*$. 
    \item The vertices with first coordinate $N$ are $a$-invariant. 
    \item The 0 and 1 labels give a partial 2-coloring of each $a$-orbit. 
\end{itemize}

\begin{claim}\label{claim:is_possible}
    $\Pi' \in \BC_{\beta+1}(\Z * \Gamma)$. 
\end{claim}

\begin{proof}
    By hypothesis, there is a Baire-class-$\beta$ $c:F(\Z * \Gamma,2^\omega) \to \Lambda$ which is a $\Pi$-labeling with respect to the action of $\Gamma$. We will use $c$ for our second coordinates and construct a Baire-class-$(\beta+1)$ function $d : F(\Z * \Gamma,2^\omega) \to \{0,1,N,*\}$ so that $(d,c)$ is a $\Pi'$-labeling. 

    We are forced to put $d(x) = *$ if and only if $c(x) \in \Lambda_*$, so $d\inv(*) \in \Del^0_{1+\beta}$. Put $d(x) = N$ if and only if $\langle a\rangle x \subseteq c\inv(\Lambda \setminus \Lambda_*)$. Then $d\inv(N) \in \boldPi^0_{1+\beta}$. 

    Define the sets $L_{\pm}^{i}$ exactly as in the proof of Claim \ref{claim:base_case_bc1}. Each of these sets is a countable union of $\Del^0_{1+\beta}$ sets, hence $\Sig^0_{1+\beta}$. Thus we can set $d\inv(i) = L_-^i \cup (L_+^i \setminus L_-)$ as before. 
\end{proof}

The following shows that we can take $\Lambda_*' = \{N\} \times \Lambda$.

\begin{claim}
    Let $(d,c) : F(\Z * \Gamma,2^\omega) \to \{0,1,N,*\} \times \Lambda$ be a Borel $\Pi'$-labeling. $d\inv(N)$ is $\boldPi^0_{1+\alpha+1}$-hard. 
\end{claim}

\begin{proof}
    We will work on $F(\Z * \Gamma,2^\omega \times 2^\omega)$ instead. Let $U_{1+\alpha}$ be the $\Sig_{1+\alpha}^0$-complete set from Theorem \ref{thm:2_color_complexity}. Let
    \[ A = \{x \mid c(x) \notin \Lambda_* \Leftrightarrow x(1)_1 \in U_{1+\alpha} \}. \]
    We will apply Marks' Lemma to this $A$.

    In the second case $c \circ f$ is a $\Pi$-labeling of $F(\Gamma,2^\omega)$, and by definition of $A$ we have $(c \circ f)(x) \in \Lambda_* \Leftrightarrow (f(x)(1))_1 \in U_{1+\alpha}$. Since $f$ and the projection maps are continuous, we have $(c \circ f)\inv(\Lambda_*) \leq_W U_{1+\alpha}$, contradicting the assumption that the former is $\boldPi^0_{1+\alpha}$-hard. 

    Therefore we are in the first case. The Borel function $d \circ f : F(\Z,2^\omega) \to \{0,1,N,*\}$ has the following properties:
    \begin{itemize}
        \item A point $x$ is labeled with $*$ if and only if $x(1) \not\in U_{1+\alpha}$. (This uses the ``moreover'' from Marks' Lemma). 
        \item The set of points labeled $N$ is $\Z$-invariant. 
        \item 0 and 1 give a partial 2-coloring. 
    \end{itemize}
    Thus, each $\Z$-orbit either contains a point labeled $*$, is entirely 2-colored by 0 and 1, or is entirely labeled by $N$. 
    Moreoever the first case occurs if and only if some vertex in the orbit has a label from $2^\omega \setminus U_{1+\alpha}$. 
    These cases are shown in Figure \ref{fig:inductive_hardness}.
    
    Let $B \subseteq F(\Z,2^\omega)$ be the set of points $x$ whose entire orbit is labeled by 0 and 1. 
    Then $B \in \mathcal{I}$ by definition. Observing that $(d \circ f)\inv(N) = F(\Z,2^\omega) \cap (U_{1+\alpha}^\Z \setminus B) = F(\Z,U_{1+\alpha}) \setminus B$, we conclude that $(d \circ f)\inv(N)$ is $\boldPi^0_{1+\alpha+1}$-hard by definition of $U_{1+\alpha}$, and therefore $d\inv(N)$ is $\boldPi^0_{1+\alpha+1}$-hard since $f$ is continuous. 
\end{proof}

\begin{figure}
    \centering
    \begin{tikzpicture}[nodes={circle, draw, minimum size=5mm, inner sep=0pt}]
    \graph [nodes={circle, draw, minimum size=5mm, inner sep=2pt, empty nodes},
            edges={-}]{
     1[label=center:\small$*$, fill = red!50] --
     2[label=center:\small$0$, fill = green!50] --
     3[label=center:\small$1$, fill = green!50] -- 
     4[label=center:\small$0$, fill = green!50] -- 
     5[label=center:\small$*$, fill = red!50] -- 
     6[label=center:\small$1$, fill = green!50] -- 
     7[label=center:\small$*$, fill = red!50];

     21[label=center:\small$0$, fill = green!50] --
     22[label=center:\small$1$, fill = green!50] --
     23[label=center:\small$0$, fill = green!50] -- 
     24[label=center:\small$1$, fill = green!50] -- 
     25[label=center:\small$0$, fill = green!50] -- 
     26[label=center:\small$1$, fill = green!50] -- 
     27[label=center:\small$0$, fill = green!50];
            
    11[label=center:\small$N$, fill = green!50] --
     12[label=center:\small$N$, fill = green!50] --
     13[label=center:\small$N$, fill = green!50] -- 
     14[label=center:\small$N$, fill = green!50] -- 
     15[label=center:\small$N$, fill = green!50] -- 
     16[label=center:\small$N$, fill = green!50] -- 
     17[label=center:\small$N$, fill = green!50];
    
     };

    \node[draw = none] (A) at (6.4,-0.6) {};
    \node[draw = none] (B) at (6.4,-2.4) {};
    \node[draw = none] (C) at (-0.4,-1.4) {};
    \node[draw = none] (D) at (-0.4,-2.6)  {};
    \node[draw = none] (E) at (-0.4,-0.4) {};
    \node[draw = none] (F) at (-0.4,-1.6)  {};

    \draw[thick, decorate,decoration={brace,amplitude=10pt}] (A) -- (B) node[draw = none, midway, xshift = 25pt]{\small$U_{1+\alpha}^\Z$};

    \draw[thick, decorate,decoration={brace,amplitude=10pt,mirror}] (C) -- (D) node[draw= none, midway, xshift = -42pt]{\small$(d \circ f)\inv(N)$};

    \draw[thick, decorate,decoration={brace,amplitude=10pt,mirror}] (E) -- (F) node[draw= none, midway, xshift = -30pt]{\small$B$};
    
     \end{tikzpicture}
    \caption{The three ways a $\Z$-orbit can look after pulling back the $\Pi'$-labeling via $f$. Points are colored red if and only if they are labeled with $*$ in order to maintain the analogy with the base case and Figure \ref{fig:base_case}.}
    \label{fig:inductive_hardness}
\end{figure}
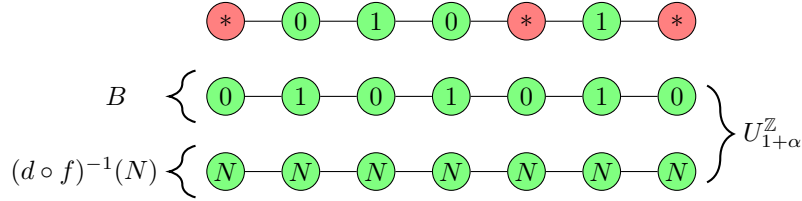

\begin{remark}
    Marks' Lemma is proved by a Borel determinacy argument. A certain game is concocted where the players define a point of $F(\Delta * \Gamma, 2^\omega \times 2^\omega)$ and compete on whether or not in lies in $A$. The idea of the proof above is to combine this game with the one used in the proof of Wadge determinacy. Indeed, the winning condition from the Wadge game is used to define the set $A$ to which Marks' Lemma is applied. 
\end{remark}

\section{Questions}

There are many more questions than answers concerning the classes $\BC_\alpha(\Gamma)$.
Here we collect the ones we find most interesting. We start by repeating Question \ref{q:main} from the introduction.

\begin{question}
    What is the least $\alpha < \omega_1$ such that $\BC_\alpha(\F_2) = \BOREL(\F_2)$?
\end{question}

Besides possibly the free groups, the groups which have received the most attention within the framework of LCLs on groups are, natrually, the abelian free groups $\Z^n$ \cite{GR_grids, BBLW}. Gao, Jackson, Krohne, and Seward showed that, for $n > 1$ and with respect to the usual generating set, proper 4-coloring is in $\CONT(\Z^n)$ while proper 3-coloring is in $\BOREL(\Z^n) \setminus \CONT(\Z^n)$ \cite{gjks, GJKS_borel}. It can be verified using their construction that, in fact, proper 3-coloring is in $\BC_1(\Z^n)$ \cite{UVW}. 
It is still unclear whether the hierarchy $\BC_\alpha(\Z^n)$ goes any farther than this.

\begin{question}
    Let $n > 1$. Is $\BOREL(\Z^n) = \BC_1(\Z^n)$?
\end{question}

As was mentioned in the introduction, for $n = 1$ we actually have $\BOREL(\Z) = \CONT(\Z)$. 

There is no real understanding of the complexity landscape for an arbitrary group $\Gamma$. We believe this deserves to be remedied, and in particular it would be interesting to know whether a more complicated group $\Gamma$ can help the hierarchy go farther than it does for $\F_n$.  

\begin{question}
    Is there a countable group $\Gamma$ for which $\BOREL(\Gamma) \neq \BC_\omega(\Gamma)$?. Is there, for every $\alpha < \omega_1$, a countable group $\Gamma$ for which $\BOREL(\Gamma) \neq \BC_\alpha(\Gamma)$? 
\end{question}

Finally, it would be interesting to know whether projective predicates can be encoded by local problems in the same way that arithmetic ones were in this paper. This line of questioning is to our knowledge completely unexplored. 

\begin{question}
    Assume Projective Determinacy. Is there an LCL $\Pi$ on $\F_2$ such that $F(\F_2,2^\omega)$ admits a projective $\Pi$-labeling but not a Borel one? What about for an arbitrary countable group $\Gamma$?
\end{question}

\section*{Acknowledgments}

We are thankful to Andrew Marks for much encouragement and many helpful conversations, especially related to useful directions for proving Theorem \ref{thm:2_color_complexity}. Thanks also to Forte Shinko and Anton Bernshteyn for interesting initial discussions regarding the main questions of this paper. 
The author is supported by the NSF under award number DMS-2402064.

\bibliographystyle{amsalpha} 
\bibliography{main}

\end{document}